\documentclass[leqno,twoside, 12pt]{amsart}
\usepackage{amssymb, latexsym, esint}
\usepackage{color}

\usepackage{amsthm}
\usepackage{amssymb,amsfonts}
\usepackage[all,arc]{xy}
\usepackage{enumerate}
\usepackage{mathrsfs}
\usepackage{amsmath}
\usepackage{verbatim}
\usepackage{hyperref}

\def\R{\mathbb R}

\theoremstyle{plain}
\numberwithin{equation}{section} \numberwithin{figure}{section}

\newtheorem{theorem}{Theorem}[section]
\newtheorem{lemma}[theorem]{Lemma}
\newtheorem{proposition}[theorem]{Proposition}

\newtheorem{definition}[theorem]{Definition}
  \usepackage{epsfig} 
\usepackage{graphicx}  \usepackage{epstopdf}
\theoremstyle{definition}
\newtheorem{remark}[theorem]{Remark}

\numberwithin{equation}{section}

\begin{document}

\markboth{Pablo Ochoa and Analia Silva}
{Magnetic fractional p-Laplacian}

%

%

\title[Calder\'on-Zygmund estimates]{Calder\'on-Zygmund estimates for higher order elliptic equations in Orlicz-Sobolev spaces}

\author{Juli\'an Fern\'andez Bonder, Pablo Ochoa and Anal\'ia Silva}

\address[J.F. Bonder]{Instituto de C\'alculo -- CONICET and
Departamento de Matem\'atica, FCEN -- Universidad de Buenos Aires. 
Ciudad Universitaria, Edificio $0+\infty$, C1428EGA, Av. Cantilo s/n
Buenos Aires, Argentina\\
{\tt jfbonder@dm.uba.ar\\ 
https://mate.dm.uba.ar/\~{}jfbonder/index.html}}

\address[P.Ochoa]{Universidad Nacional de Cuyo-CONICET, Parque Gral. San Mart\'in\\
Mendoza, 5500, Argentina\\
{\tt ochopablo@gmail.com}}

\address[A.Silva]{Departamento de Matem\'atica, FCFMyN, Universidad Nacional de San Luis and Instituto de Matem\'atica Aplicada San luis (IMASL),
UNSL-CONICET. Ejercito de los Andes
950, D5700HHW,San Luis,
Argentina\\
{\tt acsilva@unsl.edu.ar\\analiasilva.weebly.com}}

\subjclass[2020]{46E30, 35P30, 35D30}

\keywords{Orlicz-Sobolev spaces, biharmonic Laplacian}
\maketitle
\begin{abstract}
In this paper we obtain Calder\'on-Zygmund estimates for the laplacian of the following fourth order quasilinear elliptic problem
$$
\Delta(g(\Delta u)\Delta u) = \Delta(g(\Delta f)\Delta f).
$$
where the primitive of $g(t)t$, $G(t)$, is an $N-$function. We prove that if $G(f)\in L^q$, then $G(\Delta u)\in L^q$ for $q\ge 1$.
\end{abstract}

\section{Introduction}
In this paper, we consider  Calder\'on-Zygmund regularity results associated to quasilinear equations involving  higher-order operators in the framework of Orlicz-Sobolev spaces. More precisely, we are concern with higher integrability properties on weak solutions to the problem
\begin{equation}\label{eq}
\Delta(g(\Delta u)\Delta u) = \Delta(g(\Delta f)\Delta f),
\end{equation}
in a bounded domain $\Omega\subset \R^n$. Our results are of local nature so no regularity on the boundary of $\Omega$ nor boundary conditions are assumed.

This article can be seen as a first step in extending classical Calder\'on-Zygmund--type estimates to higher order nonlinear elliptic problems. For second order problems, this has been the subject of active research in the past two decades (or even more). For instance, in \cite{YZ} the authors consider weak solutions to the problem
$$
\text{div}(g(|\nabla u|)\nabla u) = \text{div}(g(|\mathbf f|)\mathbf f)
$$
and obtained the estimate
$$
G(|\mathbf f|)\in L^q \Rightarrow G(|\nabla u|)\in L^q,\qquad q\ge 1.
$$
Here $G(t)=\int_0^t g(s)s\, ds$. The results in \cite{YZ} can be thought as an extension of the classical estimates of Calder\'on-Zygmund type that for the case of the $p-$laplacian (i.e. $g(t)=t^{p-2}$) were obtained in \cite{DB-M} and \cite{I}.

So, the main result in this paper is to obtain the following estimate for weak solutions to \eqref{eq}:
$$
G(|f|)\in L_{loc}^q \Rightarrow G(|\Delta u|)\in L_{loc}^q,\qquad q\ge 1.
$$

One important motivation to analyze problem \eqref{eq} came from the analysis of Lane-Emden type systems of the form
\begin{equation}\label{system}
\begin{cases}-\Delta u = \phi(v) \quad \text{in }\Omega\\
-\Delta v = \psi(u)\quad \text{in }\Omega\\
u=v=0\quad \text{on }\partial\Omega.
\end{cases}
\end{equation}

Let us suppose that $\psi$ is odd and one to one. Let $h=\psi^{-1}$. From the equation $-\Delta v = \psi(u)$, we may solve for $u$ and get
$$
u=-h(\Delta v).
$$
Plugging into the first equation of the system \eqref{system}, we deduce
$$
\phi(v)= -\Delta u = \Delta(h(\Delta v)) = \Delta(g(\Delta v)\Delta v),
$$
where
$$
g(t)=\dfrac{h(t)}{t}.
$$ 
The operator
$$\Delta_g^2 u= \Delta \left(g(\Delta u)\Delta u\right),$$is called the  biharmonic $g$-Laplacian and it was largely studied in \cite{OS}.

\bigskip

Therefore, our main  theorem is the following:
\begin{theorem}\label{main theorem}
Let $\Omega\subset \mathbb{R}^n$ be a bounded domain. Let  $G(t)=\int_0^tg(s)s\,ds$ be an N-function satisfying \eqref{g}, \eqref{G} and \eqref{GG}-\eqref{GGG}, and let $u\in W_{loc}^{2, G}(\Omega)$ be a local weak solution of problem \eqref{eq}. 

Then, if
$$
G(f)\in L_{loc}^q(\Omega), \quad \text{for some }q \geq 1, \text{ then we also have}\quad G(\Delta u)\in  L_{loc}^q(\Omega).
$$
Moreover,  the following estimate holds: there is $C>0$ such that
\begin{equation*}
\int_{B_r(x_0)}\left(G(\Delta u)\right)^q\,dx \leq C\left[\left(\int_{B_{2r}(x_0)}G\left(\dfrac{1}{r^2} u\right)\,dx+\int_{B_{2r}(x_0)}G\left(\dfrac{1}{r} |\nabla u|\right)\,dx\right)^q + \int_{B_{2r}(x_0)}\left(G(f)\right)^q\,dx\right],
\end{equation*}for any $r>0$ and $x_0\in \Omega$ so that $B_{4r}(x_0)\subset \Omega$. 
\end{theorem}

Let us mentioned that the strategy of the proof in obtaining Theorem \ref{main theorem} were used in \cite{YZ} and that these ideas were already present in the seminal work \cite{W}.

To end this introduction, let us present an heuristic argument for the validity of Theorem \ref{main theorem}. 

Let $u\in W_{loc}^{2, G}(\Omega)$ be a local weak solution of \eqref{eq} and assume that $G(f)\in L_{loc}^q(\Omega)$. Then, the function
  $$w=g(\Delta u)\Delta u-g(f)f\in L_{loc}^1(\Omega)$$is a very weak solution of $\Delta w=0$ in $\Omega$. Then, by  \cite{ZB}, $w\in W_{loc}^{2, p}(\Omega)$ for any $p$, and consequently, $w$ is indeed a weak solution of $\Delta w=0$. This implies that 
  $$w\in C^{\infty}(\Omega).$$So, letting $h$ be the inverse of $g(t)t$,
\begin{equation}\label{heu}
 \Delta u=h(g(f)f + w),\quad w\in C^\infty,
\end{equation}and if 
  $$h(g(f)f + w)\approx f + \text{smooth term},$$we will certainly derive from \eqref{heu} that
  $$G(\Delta u)\approx G(f)\in L_{loc}^q(\Omega).$$

\subsection*{Organization of the paper}

In Section \ref{prelim} we collect the preliminaries needed in the work and state precisely the hypotheses needed for the nonlinearity $g$. In Section \ref{sec q=1} we analyze Theorem \ref{main theorem} in the special and simpler case of $q=1$. Finally, in Section \ref{sec main} we prove Theorem \ref{main theorem}.

\section{Preliminaries}\label{prelim}

In this section, we introduce basic definitions and preliminary results related to Orlicz spaces. 

\subsection*{N-functions} We start recalling the definition of an N-function.
\begin{definition}\label{d2.1}
A function $G \colon [0, \infty) \rightarrow \mathbb{R}$ is called an N-function if it admits the representation
$$
G(t)= \int _{0} ^{t} g(\tau)\tau d\tau,
$$
where the function $g$ is right-continuous for $t \geq 0$, even,  positive for $t >0$, non-decreasing and satisfies the conditions
$$
g(0)=0, \quad \lim_{t \to \infty}g(t)t=\infty.
$$
\end{definition}
By \cite[Chapter 1]{KR}, an N-function has also the following properties:
\begin{enumerate}
\item $G$ is continuous, convex, increasing, even and $G(0) = 0$.
\item $G$ is super-linear at zero and at infinite, that is 
$$
\lim_{t\rightarrow 0+} \frac{G(t)}{t}=0\quad \text{and}\quad \lim_{t\rightarrow \infty} \frac{G(t)}{t}=\infty.
$$
\end{enumerate}

Indeed, the above conditions serve as an equivalent definition of N-functions.

An important class of N-functions is the following:
\begin{definition}\label{def delta2}
We say that the N-function  $G$ satisfies the (global) $\bigtriangleup_{2}$ condition if  there exists $C > 2$ such that
\begin{equation*}
G(2t) \leq C G(t), \qquad \text{for all } t>0.
\end{equation*}
\end{definition}
			
Examples of functions satisfying the $\bigtriangleup_{2}$ condition are, for instance:
\begin{itemize}
        \item $G(t)= t^{p}$, $t \geq 0$, $p > 1$;
        \item $G(t)=(1+|t|)\log(1+|t|) - |t|$;
        \item $G(t)=t^{p}\chi_{(0, 1]}(t) + t^{q}\chi_{(1, \infty)}(t)$, $t \geq 0$, $p, q > 1$.
        \item $G(t)=t^p \log^{\alpha}(1+t)$, $p>1$, $0<\alpha\le 1$.
    \end{itemize}

    
    By \cite[Theorem 4.1, Chapter 1]{KR},  an N-function  satisfies the $\bigtriangleup_{2}$ condition if and only if there is $p^{+} > 1$ such that
			\begin{equation}\label{eq p mas}
		\frac{t^2g(t)}{G(t)} \leq p^{+}, ~~~~~\forall\, t>0.
	\end{equation}
	
Associated to $G$ is  the N-function  complementary to it which is defined as follows:
\begin{equation}\label{Gcomp}
\widetilde{G} (s) := \sup \left\lbrace ts-G(t) \colon t>0 \right\rbrace .
\end{equation}
It is shown in \cite{KR} that $\widetilde{G}$ is also an $N-$function.	
	
The definition of the complementary function gives the optimal function such that the following Young-type inequality holds
	\begin{equation}\label{2.5}
		st \leq G(t)+\widetilde{G} (s) \text{  for every } s,t \geq 0.
	\end{equation}

\begin{remark}\label{igualdad}
It is easy to check that the equality in \eqref{2.5} holds for  $s=g(t)(t)$.
\end{remark}
	
We also quote the following useful lemma.
\begin{lemma}\cite[Lemma 2.9]{BS}\label{G g}
Let $G$ be an N-function. If $G$ satisfies \eqref{eq p mas} then 
        \[
            \tilde{G}(g(t)t) \leq (p^+-1)G(t).
        \]
\end{lemma}
	
By \cite[Theorem 4.3,   Chapter 1]{KR}, a necessary and sufficient condition for the N-function $\widetilde{G} $ complementary to $G$ to satisfy the $\bigtriangleup_{2}$ condition is that there is $p^{-} > 1$ such that
\begin{equation}
p^{-} \leq 	\frac{t^2g(t)}{G(t)}, ~~~~~\forall\, t>0.
\end{equation}

From now on, we will assume the stronger condition that the N-function  $G(t)= \int _{0} ^{t} g(s) s\, ds$  satisfies the following estimate:
\begin{equation}\label{g}
p^{-}-3 \leq \frac{tg''(t)}{g'(t)} \leq p^{+}-3, \quad \forall t>0,
\end{equation}
for some $1<p^-\le p^+$. Moreover we assume that
\begin{equation}\label{G}
G(\sqrt{t}) \text{ is convex in }\mathbb{R}.
\end{equation}
	
	We observe that condition \eqref{g} implies
\begin{equation}\label{g1}
p^{-}-2 \leq \frac{tg'(t)}{g(t)} \leq p^{+}-2
\end{equation}and
	\begin{equation}\label{G1}
1<p^{-} \leq \frac{t^2g(t)}{G(t)} \leq p^{+},
\end{equation}and so the $\Delta_2$-condition holds for both $G$ and $\tilde{G}$. Also, observe that when $g(t)=|t|^{p-2}$, then assumption \eqref{G} implies $p \geq 2$. 

Observe that the examples presented after Definition \ref{def delta2} all verify \eqref{g}.

As a consequence of \eqref{G1}, it follows  for $s, t \geq 0$ that
\begin{equation}\label{ineq at}
\min\left\lbrace s^{p^+}, s^{p^-}\right\rbrace G(t)\leq G(st)\leq 	G(t)\max\left\lbrace s^{p^+}, s^{p^-}\right\rbrace.
\end{equation}
Hence, combining \eqref{2.5} and \eqref{ineq at}, we have the following Young's inequality with $\varepsilon\in (0, 1)$:
\begin{equation}\label{Young}
st = (\varepsilon s) (\varepsilon^{-1}t) \leq \widetilde{G}(\varepsilon s)+G(\varepsilon^{-1}t ) \leq \varepsilon \widetilde{G}(s)+ \varepsilon^{-p^+}G(t).
\end{equation}

We quote the following useful  estimate (see \cite[Lemma 7.1]{O}):

\begin{lemma}\label{inequality g}
Suppose that $G$ is an N-function satisfying \eqref{G} and \eqref{g1}. There exists a constant $C>0$ such that for all $a, b \in \mathbb{R}$ we have
$$
(g(a)a-g(b)b)(a-b) \geq CG(|a-b|).
$$
\end{lemma} 	
	
\begin{definition}
Given two $N$-functions $A$ and $B$, we say that $A$ increases essentially more slowly than $B$, denoted by $A \ll B,$ if for any $c > 0$,
\begin{equation*}
\lim_{t \rightarrow \infty} \dfrac{A(ct)}{B(t)}=0.
\end{equation*}
\end{definition}

\subsection*{Orlicz-Sobolev spaces} Given an N-function $G$, with $G'(t)=g(t)t$, we define the Orlicz-Lebesgue class $L^G(\Omega)$ as follows
$$
L^G(\Omega):=\left\lbrace u: \Omega \to \mathbb{R}, \int_\Omega G(u)\,dx < \infty\right\rbrace.
$$
If $G$ satisfies the $\Delta_2$ condition, then $L^G$ becomes a vector spaces, that it is a Banach space with respect, for instance, to the Luxemburg norm
$$
\|u\|_G:=\inf \left\lbrace \lambda > 0: \int_\Omega G\left(\dfrac{u}{\lambda}\right)\,dx \leq 1\right\rbrace.
$$
We will denote the convex modular associated to the norm by
$$
\rho_{G}(u):=\int_\Omega G(u)\,dx.
$$

For any positive integer $m\in \mathbb N$ we will also consider the Orlicz-Sobolev spaces
$$
W^{m, G}(\Omega):=\left\lbrace u \in L^G(\Omega), \, D^\alpha u  \in L^G(\Omega), \,\text{for all multi-index }\, |\alpha|\leq m\right\rbrace.
$$
The space $W^{m, G}(\Omega)$ equipped with the norm
$$
\|u\|_{m,G}:= \sum_{|\alpha|\leq m}\|D^{\alpha}u\|_G
$$
is a Banach space. 

In the case where also $\widetilde G$ satisfies the $\Delta_2$ condition, these spaces are reflexive and separable.

In order to have the Sobolev immersions, we need to require some assumptions on $G$ that are the analog of $p<n$ in the classical cases. So, we need to assume that $G$ satisfies:
\begin{equation}\label{GG}
\int_0^1 \dfrac{G^{-1}(s)}{s^{(n+1)/n}}\,ds <\infty
\end{equation}

\begin{equation}\label{GGG}
\int_1^\infty \dfrac{G^{-1}(s)}{s^{(n+1)/n}}\,ds =\infty
\end{equation}

 For a given $N$-function $G$, define the first order Sobolev conjugate function $G_1^{*}$  of $G$ by means of
 $$(G_1^{*})^{-1}(t):=\int_0^{t}\dfrac{G^{-1}(s)}{s^{1+1/n}}\,ds.$$Then $G_1^*$ is an $N$-function (see \cite{BSV2}). Next, we define the $m$-th order conjugate Sobolev function of $G$  recursively as follows
 $$
 G^{*}_0 := G,\quad G^{*}_j:= (G^{*}_{j-1})^*, \quad j=1,\dots, m.
 $$
 At each stage, we assume that
 $$\int_0^1 \dfrac{(G_j^*)^{-1}(s)}{s^{(n+1)/n}}\,ds <\infty.$$We obtain in this way a finite sequence of $N$-functions $G^{*}_j$, $j=0,\dots, m_0$, where $m_0$ is such that
 $$\int_1^\infty \dfrac{(G_{m_0-1}^*)^{-1}(s)}{s^{(n+1)/n}}\,ds =\infty$$but
 $$\int_1^\infty \dfrac{(G_{m_0}^*)^{-1}(s)}{s^{(n+1)/n}}\,ds <\infty.$$Indeed, $m_0 \leq n$, since by induction it can be proved that (see \cite{DT})
 $$(G_m^*)^{-1}(t)\leq K_m t^{(n-j)/n}.$$

 Then, we have the following embedding theorem for higher-order Orlicz-Sobolev spaces stated in \cite{DT}.

 \begin{theorem}\label{compact theorm}
 Let $\Omega \subset \mathbb{R}^{n}$ be a bounded domain with the cone property. Let $G$ be an $N$-function satisfying \eqref{G1} and let $m_0$ be defined as before. Then
 \begin{enumerate}
 \item if $1 \leq m \leq m_0$, then $W^{m, G}(\Omega) \hookrightarrow L^{G_{m}^{*}}(\Omega)$. Moreover, if $B$ is an $N$-function increasing essentially more slowly than $G^*_m$ near infinity, then the embedding $W^{m, G}(\Omega) \hookrightarrow L^{B}(\Omega)$ is compact;
 
 \item if $m > m_0$, then $W^{m, G}(\Omega) \hookrightarrow C(\Omega) \cap L^{\infty}(\Omega)$.
 \end{enumerate}
 \end{theorem}

In this paper we will consider the second order case $m=2$. 

For the construction of some auxiliary problems, we need to take into account  some boundary conditions. To this end, we denote $W^{2, G}_0(\Omega)$ the closure of $C_0^{\infty}(\Omega)$ in $W^{2, G}(\Omega)$. 

By \cite{WW},  the norm $\|u\|_{2, G}$ in $W^{2, G}_0(\Omega)$ is equivalent to
$$
\|u\|_{2, G}\sim \|\Delta u\|_G.
$$

From now, we will consider the norm $\|\Delta u\|_{G}$ in the space $W^{2, G}_0(\Omega)$.  The relevant modular defined on $W^{2, G}_0(\Omega)$ is given by
$$
\rho_{2,G}(u):=\rho_{G}(\Delta u)=\int_\Omega G(\Delta u)\,dx.
$$

To close the section, we quote the following further  useful relation between modulars and norms.
\begin{lemma}\label{comp norm modular}
        Let $G$ be an N-function satisfying \eqref{G1},  and let 
        $\xi^\pm\colon[0,\infty)$ $\to\mathbb{R}$ be defined as
        \[
            \xi^{-}(t):= 
            \min \big \{  t^{p^{-}}, t^{p^{+}} \big  \} ,
            \quad \text{ and }  \quad
            \xi^{+}(t):=\max \big \{  t^{p^{-}}, t^{p^{+}} \big \} . 
         \] 
         Then
   
             $$\xi^{-}(\|u\|_{2, G}) \leq \rho_{2,G}( u) 
                    \leq  \xi^{+}(\|u\|_{2, G}).$$
    \end{lemma}

\subsection*{Concept of solution}
Let us state the definition of solution that will be used throughout the paper:
\begin{definition}\label{defi weak}We say that $u\in W_{loc}^{2, G}(\Omega)$ is a local weak solution of \eqref{eq} in $\Omega$ if and only if
$$\int_\Omega g(\Delta u)\Delta u\Delta \varphi\,dx =\int_\Omega g(f)f\Delta \varphi\,dx, $$for any $\varphi\in C_{0}^{\infty}(\Omega)$. 
\end{definition}

\begin{remark}Observe that by density, we may take $\varphi$ in $W_0^{2, G}(\Omega)$ in Definition \ref{defi weak}. 
\end{remark}
  
\subsection*{Maximal functions and a Vitali Lemma}
In this section, we will include the properties of the maximal functions that we employ. Moreover, we state a modified Vitali Lemma from \cite{W}.   

 \begin{definition}Let $f\in L^1_{loc}(\Omega)$. The Hardy-Littlewood maximal function denoted by $\mathcal{M}(f)$ is defined by
 $$\mathcal{M}(x)=\sup_{r>0}\fint_{B_r(x)}f(x)\,dx.$$ 
 \end{definition}The following are standard properties of the maximal function: given $f\in L^q(\Omega), q >1,$
 \begin{itemize}
 \item $\|\mathcal{M}(f)\|_{L^q(\Omega)}\leq \|f\|_{L^q(\Omega)}$,
 \item for any $\lambda>0$, $$|\left\lbrace x\in \Omega: \mathcal{M}(f)(x)> \lambda\right\rbrace| \leq \dfrac{C}{\lambda}\int_{\Omega}|f(x)|\,dx,$$
 \end{itemize}
 
 We quote a modified version of the Vitali Lemma from \cite{W}:
 
 \begin{lemma}\label{vitali}
 Let $\varepsilon\in (0, 1)$ and let $C\subset D\subset B_1$ be two measurable sets with $|C|< \varepsilon |B_1|$ and satisfying the following property: 
$$
\text{for every  $x\in B_1$ with } |C\cap B_{r}(x)|\geq \varepsilon|B_r|, \text{ there holds } B_r(x)\cap B_1\subset D.
$$ 
Then,
$$
|C|\leq 10^n\varepsilon|D|.
$$
 \end{lemma}
 
 Finally, we state an iteration lemma from \cite{Gi}:
 \begin{lemma}\label{iteration}
 Assume that $\varphi$ is a non-negative,
real-valued, bounded function defined on an interval $[r, R] \subset \mathbb{R}_+$. Assume
further that for all $r\leq \rho<\sigma \leq R$  we have
$$\varphi(\rho)\leq A_1(\sigma-\rho)^{-\alpha_1}+A_2(\sigma-\rho)^{-\alpha_2}+A_3 + \theta \varphi(\sigma),$$for $\theta \in (0, 1)$ and positive constant $A_1$, $A_2$, $A_3$ and $\alpha_1\leq \alpha_2$. Then, there is a constant $C(\alpha_2, \theta)>0$ such that
$$\varphi(r)\leq C(\alpha_2, \theta)\left(A_1(R-r)^{-\alpha_1}+A_2(R-r)^{-\alpha_2}+A_3 \right)$$
 \end{lemma}

\section{Calder\'on-Zygmung regularity: the case $q=1$}\label{sec q=1}
   
We will start by proving the main result for $q=1$.
\begin{theorem}\label{theorem q 1}Let $G$ be an $N$-function satisfying \eqref{g} and \eqref{G}.   Let $u\in W_{loc}^{2, G}(\Omega)$ be a local weak solution of \eqref{eq}. 

Then, there is $C>0$ such that
\begin{equation}\label{estimete for q 1}
\int_{B_r(x_0)}G(\Delta u)\,dx \leq C\left(\int_{B_{2r}(x_0)}G\left(\dfrac{1}{r^2} u\right)\,dx+\int_{B_{2r}(x_0)}G\left(\dfrac{1}{r}|\nabla u|\right)\,dx + \int_{B_{2r}(x_0)}G(f)\,dx\right),
\end{equation}
for any $r$ and $x_0$ such that $B_{4r}(x_0)\subset \Omega$.
\end{theorem}

\begin{proof} We will prove the estimate \eqref{estimete for q 1} for $r=1$. The result for an arbitrary ball $B_r(x_0)\subset B_{4r}(x_0)\subset \Omega$ will follows by making the change of variables
$$
u_r(x)=\dfrac{1}{r^2}u(rx).
$$

Let $1<\sigma< \rho < 2$ and take a cut-off function  $\xi \in C_0^{\infty}(B_\rho)$, $\xi = 1$ in $B_\sigma$, $0\leq \xi\leq 1$ in $\mathbb{R}^n$ such that
\begin{equation}\label{est gradient laplace}
|\nabla \xi|\leq \dfrac{C}{\rho-\sigma}, \quad |\Delta \xi|\leq \dfrac{C}{(\rho-\sigma)^2}.
\end{equation}   
Then $\xi u\in W_0^{2, G}(B_2)$ and so it can be used as a test function in Definition \ref{defi weak}. Using that
$$
\nabla(\xi u)= \xi \nabla u + u\nabla \xi \quad \text{ and }\quad \Delta (u \xi)= \xi \Delta u + 2 \nabla u \nabla \xi + u\Delta \xi,
$$
we obtain
\begin{equation}\label{calc}
\begin{split}
\int_{B_2}\xi g(|\Delta u|)|\Delta u|^2\,dx =& -\int_{B_2}ug(\Delta u)\Delta u\Delta \xi\,dx -2\int_{B_2}g(\Delta u)\Delta u\nabla u \cdot\nabla \xi\,dx \\ &  +\int_{B_2}\xi g(f)f\Delta u\,dx + \int_{B_2}ug(f)f\Delta \xi\,dx +2\int_{B_2}g(f)f\nabla u\cdot \nabla \xi \,dx\\
&= I + II + III + IV + V. 
\end{split}   
\end{equation}Since $\xi=1$ in $B_\sigma$, $\xi \in C_0^{\infty}(B_\rho)$ and by \eqref{G1}, we have
\begin{equation*}
p^{-}\int_{B_\sigma}G(\Delta u)\,dx \leq \int_{B_\rho}\xi g(|\Delta u|)|\Delta u|^2\,dx.
\end{equation*}In the following estimates we appeal to  Young's inequality with $\varepsilon>0$, see \eqref{Young}, the bounds \eqref{est gradient laplace} and Lemma \ref{G g}:
\begin{align*}
|I| = \left|\int_{B_\rho}ug(\Delta u)\Delta u\Delta \xi\,dx\right| \leq& \varepsilon\int_{B_\rho}\tilde{G}(g(\Delta u)\Delta u)\,dx+C_\varepsilon(\rho-\sigma)^{-2p^+}\int_{B_\rho}G(u)\,dx\\ 
\leq & C\varepsilon \int_{B_\rho}G(\Delta u)\,dx+C_\varepsilon(\rho-\sigma)^{-2p^+}\int_{B_\rho}G(u)\,dx;
\end{align*}
$$
|II| = \left|\int_{B_\rho}g(\Delta u)\Delta u\nabla u \cdot\nabla \xi\,dx\right| \leq C\varepsilon \int_{B_\rho}G(\Delta u)\,dx+C_\varepsilon(\rho-\sigma)^{-p^{+}}\int_{B_\rho}G(|\nabla u|)\,dx;
$$
\begin{align*}
|III| &= \left|\int_{B_\rho}\xi g(f)f\Delta u\,dx\right|\leq C_\varepsilon\int_{B_\rho}G(f)\,dx + \varepsilon\int_{B_\rho}G(\Delta u)\,dx;\\
|IV| &= \left|\int_{B_\rho}ug(f)f\Delta \xi\,dx\right|\leq C\varepsilon\int_{B_\rho}G(f)\,dx + C_\varepsilon(\rho-\sigma)^{-2p^+}\int_{B_\rho}G(u)\,dx;\\
|V| &= \left|\int_{B_\rho}g(f)f\nabla u\cdot \nabla \xi \,dx\right|\leq C\varepsilon\int_{B_\rho}G(f)\,dx + C_\varepsilon(\rho-\sigma)^{-p^+}\int_{B_\rho}G(|\nabla u|)\,dx.
\end{align*}

Hence, we obtain
\begin{align*}
\int_{B_\sigma}G(\Delta u)\,dx \leq& C\varepsilon \int_{B_\rho}G(\Delta u)\,dx+ C_\varepsilon(\rho-\sigma)^{-2p^+}\int_{B_\rho}G( u)\,dx\\
& + C_\varepsilon(\rho-\sigma)^{-p^+}\int_{B_\rho}G(|\nabla u|)\,dx + C_\varepsilon\int_{B_\rho}G(f)\,dx,
\end{align*}
and so by Lemma \ref{iteration}, there is a constant $C(p^{+}, \varepsilon)>0$ such that
$$
\int_{B_1}G(\Delta u)\,dx \leq C(p^+, \varepsilon)\left(\int_{B_2}G(f)\,dx+\int_{B_2}G(u)\,dx+ \int_{B_2}G(|\nabla u|)\,dx\right).
$$
This ends the proof.
\end{proof}
   
\section{Calder\'on-Zygmung regularity: higher order integrability}\label{sec main}
   
In this section, we will prove Theorem \ref{main theorem}. We will drive the proof into some steps.
   
\subsection*{Step 1: approximating sequence and preliminary estimates}

Given $\varepsilon>0$, define
$$g_\varepsilon(t):=g(\sqrt{\varepsilon+t^2}).$$
Then, for 
$$G_\varepsilon(t):=\int_0^t g_\varepsilon(s)s\,ds,$$we have that $G_\varepsilon$ is an N-function and moreover
$$L^{G_\varepsilon}(\Omega)=L^G(\Omega)$$and 
$$W^{2, G_\varepsilon}(\Omega)=W^{2, G}(\Omega).$$

Let $B_2\subset \Omega$. Consider the following boundary value problem:
\begin{equation}\label{appox solution}
\begin{cases}
\Delta\left(g_\varepsilon(\Delta v^\varepsilon)\Delta v^\varepsilon\right)=0 \quad\text{in }B_2\\
v^\varepsilon-u\in W_0^{2, G}(B_2)
\end{cases}
\end{equation}  The existence of $v^\varepsilon$ is proved in the next proposition.

 \begin{proposition}\label{existence appox}
  Assume that $u\in W_{loc}^{2, G}(\Omega)$ is a local weak solution of \eqref{eq} in $\Omega$. Let $B_2\subset \Omega$. Then, for any $\varepsilon>0$, there is $v^{\varepsilon}\in W^{2, G}(B_2)$ solving \eqref{appox solution}  and such that
  \begin{equation}\label{uniform integral estimate}
  \int_{B_2}G_\varepsilon(\Delta v^\varepsilon)\,dx \leq C\left(\int_{B_2}G(\Delta u)\,dx +1, \right)
  \end{equation}where the constant $C>0$ is independent of $\varepsilon$. 
   \end{proposition}
\begin{proof}
For the proof we use the direct method of the calculus of variations. So we define the functional
$$
\mathcal{F}_{\varepsilon}(v)=\int_{B_2} G_\varepsilon (\Delta v)\, dx
$$
and we look for minima of $\mathcal{F}_{\varepsilon}$,
$$
\min_{v-u\in W_0^{2,G}(B_2)}\mathcal{F}_{\varepsilon}(v) 
$$
First, we show that $\mathcal{F}_{\varepsilon}$ is coersive. Observe that
\begin{align*}
G_\varepsilon(t)=\int_0^t sg_\varepsilon(s)\,ds=G(\sqrt{\varepsilon+t^2})-G(\sqrt{\varepsilon}),
\end{align*}
so
$$
G(t)\leq G(\sqrt{\varepsilon+t^2})= G_{\varepsilon}(t)+G(\sqrt{\varepsilon}).
$$
Suppose that $\|\Delta v\|_G>1$, then
$$
\mathcal{F}_{\varepsilon}(v)=\int_{B_2} G_{\varepsilon}(\Delta v)\, dx\geq \int_{B_2} G(\Delta v)\, dx \geq \|\Delta v\|_G^{p^-}.
$$
On the other hand
\begin{align*}
 \|\Delta v\|_G &= \|\Delta (v-u)+\Delta u\|_G \geq \|\Delta (v-u)\|_G-\|\Delta u\|_G\\
&\geq C\|v-u\|_{2,G} - C\geq C\|v\|_{2,G}-C.
\end{align*}
Observe that $\mathcal{F}_{\varepsilon}$ is bounded below, since $G_\varepsilon(t)\ge 0$.

Now, we want to prove that  $\mathcal{F}_{\varepsilon}$ is weakly semi continuous. Suppose that $w_k\rightharpoonup w$ in $W^{2,G}(B_2)$. Using that the set $\mathcal{A}:=\{v\in W^{2,G}(B_2)\colon  v=u \mbox{ in }\partial B_2\}$ is convex and strongly closed we obtain that $w\in\mathcal{A}$. Also, up to subsequence, $w_k\to w$ in $L^G(\Omega)$ and $w_k\to w$ a.e. By Egoroff Theorem, for any $\delta>0$ there exists $E_\delta>0$ such that $w_k\to w$ uniformly in $E_\delta$ and $|B_2\setminus E_\delta|<\delta$. We define 
$$
F_\delta=\left\{x\in B_2\colon |w(x)|+ |\Delta w(x)|\leq\frac{1}{\delta}\right\}.
$$
Then $|B_2\setminus F_\delta|\to 0$ as $\delta\to 0$. Taking $H_\delta=E_\delta\cap F_{\delta}$, observe that $|B_2\setminus H_\delta|\to 0$ as $\delta\to 0$. Now,
$$
\mathcal{F}_{\varepsilon}(w_k)=\int_{B_2} G_{\varepsilon}(\Delta w_k)\,dx\geq \int_{H_\delta} G_{\varepsilon}(\Delta w_k)\,dx.
$$
Using the convexity of $G_\varepsilon(t)$
$$
\int_{H_\delta} G_{\varepsilon}(\Delta w_k)\,dx\geq \int_{H_\delta} G_{\varepsilon}(\Delta w)\,dx+\int_{H_\delta} G_{\varepsilon}^\prime(\Delta w)(w_k-w)\,dx .
$$
As the second term in the right hand side goes to 0 as $k\to\infty$. Then, taking limit
$$
\liminf_{k\to\infty}\mathcal{F}_{\varepsilon}(w_k)\geq \int_{H_\delta} G_{\varepsilon}(\Delta w)\,dx\qquad\mbox{ for all }\delta>0.
$$
So, $\liminf_{k\to\infty}\mathcal{F}_{\varepsilon}(w_k)\geq \mathcal{F}_{\varepsilon}(w)$.

Finally, we show the uniqueness of the solution. The uniqueness of minimizers follows from the strict convexity of $G_\varepsilon$, so it remains to see that all solutions to \eqref{appox solution} are in fact minimizers of $\mathcal F_\varepsilon$. But this follows in a classical fashion, in fact using $u-v$ as a test function in the weak formulation of \eqref{appox solution}, we get
\begin{align*}
0&=\int_{B_2}\Delta(g(\Delta u)\Delta u)(u-v)\,dx=\int_{B_2} g(\Delta u)\Delta u (\Delta u-\Delta v)\,dx\\
&=\int_{B_2} g(\Delta u)\Delta u \Delta u\,dx-\int_{B_2} g(\Delta u)\Delta u\Delta v\,dx\\
&\geq\int_{B_2} g(\Delta u)\Delta u \Delta u\,dx-\int_{B_2}\tilde{G}(g(\Delta u)\Delta u)\, dx-\int_{B_2} G(\Delta v)\, dx
\end{align*}
where we have used the Young-type inequality \eqref{2.5} in the last step. Therefore, we arrive at
$$
\int_{B_2} G(\Delta v)\,dx\geq \int_{B_2} g(\Delta u)\Delta u \Delta u-\tilde{G}(g(\Delta u)\Delta u)\,dx=\int_{B_2} G(\Delta u)\,dx,
$$
where in the last equality we use Remark \ref{igualdad}.

This completes the proof.
\end{proof} 

\begin{remark}\label{Diff laplace} Observe that the function $w:=g_{\varepsilon}(\Delta v^\varepsilon)\Delta v^{\varepsilon}$ is a very weak solution of the equation $\Delta w=0$ in $B_2$ (see for instance \cite{ZB}) and hence, by \cite[Theorem 1.3]{ZB}, $w\in W^{2, p}_{loc}(B_2)$ for any $p\geq 1$. Hence, $w$ is a indeed a weak solution and therefore $w\in C^{\infty}(B_2)$. In addition, the function
$$
t\to g_\varepsilon(t)t
$$
is strictly monotone and of class $C^{2,\alpha}$. Then, by the inverse function Theorem, it has a differentiable inverse $h_\varepsilon\in C^{2,\beta}$ and since
$$
\Delta v^{\varepsilon}= h_{\varepsilon}(g_\varepsilon(\Delta v^{\varepsilon})\Delta v^{\varepsilon}),
$$
we derive that $v^{\varepsilon}$ is in $C^{4,\gamma}(B_2)$ for some $0<\gamma<1$.  We will use this fact in the next result.
\end{remark}

\begin{proposition}\label{maximum principle}
  Assume that $u\in W_{loc}^{2, G}(\Omega)$ is a local weak solution of \eqref{eq} in $\Omega$. Let $B_2\subset \Omega$. Then, for any $\varepsilon>0$, the solution $v^{\varepsilon}\in W^{2, G}(B_2)$ of \eqref{appox solution} satisfies for each $\delta\in (0,2)$,
\begin{equation}\label{max princ}
\sup_{B_{2-\delta}}G(\Delta v^\varepsilon)\leq C\left(\int_{B_2}G(\Delta u)\,dx +1\right)
\end{equation}where $C>0$ depends on $\delta$ but not on $\varepsilon$.
   \end{proposition}
\begin{proof}
Let us consider first
$$w=G_\varepsilon(\Delta v),$$where $v:=v^{\varepsilon}$. Then, for each $1\leq i \leq n $, and by Remark \ref{Diff laplace}, we may differentiate $w$ to get
$$
D_i w= g(\sqrt{\varepsilon+(\Delta v)^2})\Delta v D_i(\Delta v)
$$
and
\begin{equation}\label{Laplcian G epsilon}
\begin{split}
D_{ii}w =& \dfrac{g'(\sqrt{\varepsilon+(\Delta v)^2})}{\sqrt{\varepsilon+(\Delta v)^2}}(\Delta v)^2 (D_i(\Delta v))^2+g(\sqrt{\varepsilon+(\Delta v)^2})(D_i(\Delta v))^2\\
& +g(\sqrt{\varepsilon+(\Delta v)^2})\Delta v D_{ii}(\Delta v)\\
=:& I_1+I_2+I_3.
\end{split}
\end{equation}In the next lines, we will show that there is a uniformly elliptic coefficient $a=a(x)$ such that
\begin{equation}\label{eq for w}
a(x)\Delta w \geq 0, \text{ in }B_2,
\end{equation}
that is, $w$ is subharmonic.

Since
\begin{equation*}
D_i\left(g_\varepsilon(\Delta v)\Delta v\right)= \left[\dfrac{g'(\sqrt{\varepsilon+(\Delta v)^2})}{\sqrt{\varepsilon+(\Delta v)^2}}(\Delta v)^2 +g(\sqrt{\varepsilon+(\Delta v)^2})\right] D_{i}(\Delta v)
\end{equation*}we obtain by multiplication of equation \eqref{appox solution} by $\Delta v$ that 
$$
0 = \Delta v \Delta\left(g_\varepsilon(\Delta v)\Delta v\right) = \Delta v\sum_{i=1}^nD_{ii}\left(g_\varepsilon(\Delta v)\Delta v\right),
$$
but
\begin{equation}\label{j1}
\begin{split}
\Delta v D_{ii}\left(g_\varepsilon(\Delta v)\Delta v\right) =& \left(g''(\sqrt{\varepsilon+(\Delta v)^2}) -\dfrac{g'(\sqrt{\varepsilon+(\Delta v)^2})}{\sqrt{\varepsilon+(\Delta v)^2}}\right)\dfrac{(\Delta v)^4}{\varepsilon+(\Delta v)^2}[D_i(\Delta v)]^2\\
&+ 2\dfrac{g'(\sqrt{\varepsilon+(\Delta v)^2})}{\sqrt{\varepsilon+(\Delta v)^2}}(\Delta v)^2 [D_i(\Delta v)]^2  +\dfrac{g'(\sqrt{\varepsilon+(\Delta v)^2})}{\sqrt{\varepsilon+(\Delta v)^2}}(\Delta v)^3D_{ii}(\Delta v)\\
&+ \dfrac{g'(\sqrt{\varepsilon+(\Delta v)^2})}{\sqrt{\varepsilon+(\Delta v)^2}}(\Delta v)^2[D_i(\Delta v)]^2 +g(\sqrt{\varepsilon+(\Delta v)^2}) \Delta v D_{ii}(\Delta v)\\ 
 =:&J_1+J_2+J_3+J_4+J_5.
\end{split}
\end{equation}Observe that
\begin{equation}\label{j2}
J_2 =2I_1, J_4=I_1 \quad \text{and }\, J_5=I_3.
\end{equation}Moreover, if $\Delta v D_{ii}(\Delta v)\leq 0$, then $J_3\leq 0$ and we regret this term. So we may assume that $\Delta v D_{ii}(\Delta v) >0$.

Now, by \eqref{g},
$$g''(\sqrt{\varepsilon+(\Delta v)^2}) -\dfrac{g'(\sqrt{\varepsilon+(\Delta v)^2})}{\sqrt{\varepsilon+(\Delta v)^2}}\leq \dfrac{(p^+-4)g'(\sqrt{\varepsilon+(\Delta v)^2})}{\sqrt{\varepsilon+(\Delta v)^2}}$$and so
\begin{equation}\label{j3}
J_1 \leq (p^+-4)\dfrac{(\Delta v)^2}{\varepsilon+(\Delta v)^2}I_1.
\end{equation}Regarding $J_3$, we have
\begin{equation}\label{j4}
\begin{split}
J_3=\dfrac{g'(\sqrt{\varepsilon+(\Delta v)^2})}{\sqrt{\varepsilon+(\Delta v)^2}}(\Delta v)^3D_{ii}(\Delta v) &\leq (p^+-2)\dfrac{(\Delta v)^2}{\varepsilon+(\Delta v)^2}g(\sqrt{\varepsilon+(\Delta v)^2})\Delta v D_{ii}(\Delta v)\\&=(p^+-2)\dfrac{(\Delta v)^2}{\varepsilon+(\Delta v)^2}I_3.
\end{split}
\end{equation}

Therefore, combining \eqref{j1}, \eqref{j2}, \eqref{j3} and \eqref{j4}, using that $I_2\geq 0$  and $\dfrac{(\Delta v)^2}{\varepsilon+(\Delta v)^2}\leq 1$, we obtain
$$0 = \sum_{i=1}^5 J_i \leq \left(3+(p^+-4)\dfrac{(\Delta v)^2}{\varepsilon+(\Delta v)^2}\right) \Delta w,$$which proves \eqref{eq for w}.

Therefore, the estimate \eqref{max princ} follows by observing that $G(t)\leq G_\varepsilon(t)$ for all $t$, the weak maximum principle \cite[Chapter IV, Lemma 1.2]{CW} applied to $w=G_\varepsilon(\Delta v^{\varepsilon})$ and the uniform estimate \eqref{uniform integral estimate}. 
\end{proof}

\subsection*{Step 2: convergence of the approximating sequence}   In the next proposition, we show that the approximating sequence $v^{\varepsilon}$ converges to a biharmonic solution.
    
    \begin{proposition}
    There exists $v\in W^{2, G}(B_2)$ which is a weak solution of
    \begin{equation}\label{eq for v}
    \begin{cases} \Delta\left(g(\Delta v)\Delta v \right)=0, \text{ in }B_2,\\ v=u, \text{ on }\partial B_2,
    \end{cases}
    \end{equation}and satisfies, for each $\delta\in (0, 2)$, the estimate
    \begin{equation}\label{estimate v}
    \sup_{B_{2-\delta}} G(\Delta v)\leq C\left( \int_{B_2}G(\Delta u)\,dx +1\right).
    \end{equation}
    \end{proposition}
    
    \begin{proof}
    By the estimate \eqref{uniform integral estimate}, there is a subsequence $\varepsilon_k\to 0$, with corresponding solutions $v_k:=v^{\varepsilon_k}$ of \eqref{appox solution} and a function $v\in W^{2, G}(B_2)$, with $u-v\in W_0^{2, G}(B_2)$ such that
    $$v_k \rightharpoonup v \text{ weakly in }W^{2, G}(B_2).$$Consequently,
    $$v_k \to v \text{ strongly in } L^G(B_2)$$
    and$$\nabla v_k \to \nabla v \text{ strongly in } L^G(B_2).$$
    Let $r_m>0$ be an increasing sequence such that $r_m \to 2$ as $m \to \infty$. Fix  $m$ with $r_m>1$. Then, by Proposition  \ref{maximum principle} and the fact that $G\leq G_{\varepsilon}$,
    \begin{equation}\label{bound in B m}
   \|G(v_k)\|_{L^{\infty}(B_{r_m})}\leq  \|G_{\varepsilon_k}(v_k)\|_{L^{\infty}(B_{r_m})}\leq M_m, \text{ for all }k.
    \end{equation}Let $\varphi\in C_0^{\infty}(B_2)$, such that $0\leq \varphi \leq 1$, $\varphi= 1$ in $B_{r_m}$. In what follows, we will show that $\Delta v_k$ is a Cauchy sequence in measure in $B_{r_m}$. Hence, take $k, l$ and consider the corresponding solutions $v_k$ and $v_l$ of \eqref{appox solution}. Then,
$$
\Delta (g(\Delta v_k)\Delta v_k) = \Delta (g(\Delta v_k)\Delta v_k-g_{\varepsilon_k}(\Delta v_k)\Delta v_k)
$$
and similarly,
$$
\Delta (g(\Delta v_l)\Delta v_l) = \Delta (g(\Delta v_l)\Delta v_l-g_{\varepsilon_l}(\Delta v_l)\Delta v_l).
$$
Next, testing with $(v_k-v_l)\varphi$, we get
\begin{equation}\label{four terms}
\begin{split}
&\int_{B_2}\left[g(\Delta v_k)\Delta v_k -g(\Delta v_l)\Delta v_l\right](\Delta v_k-\Delta v_l)\varphi\,dx \\
&\qquad = -\int_{B_2}\left[g(\Delta v_k)\Delta v_k -g(\Delta v_l)\Delta v_l\right](v_k-v_l)\Delta \varphi\,dx\\
&\qquad \quad-  2\int_{B_2}\left[g(\Delta v_k)\Delta v_k -g(\Delta v_l)\Delta v_l\right](\nabla v_k-\nabla v_l)\cdot \nabla \varphi\,dx \\ 
& \qquad \quad + \int_{B_2}[g(\Delta v_k)\Delta v_k-g_{\varepsilon_k}(\Delta v_k)\Delta v_k]\Delta[(v_k-v_l) \varphi]\,dx \\
& \qquad \quad - \int_{B_2}[g(\Delta v_l)\Delta v_l-g_{\varepsilon_l}(\Delta v_l)\Delta v_l]\Delta[(v_k-v_l) \varphi]\,dx\\ 
& \qquad =:I_1+I_2+I_3+I_4.
\end{split}
\end{equation}
    Regarding the term $I_1$, observe that from Lemma \ref{G g} and Proposition \ref{existence appox}, the term
    $$g(\Delta v_k)\Delta v_k -g(\Delta v_l)\Delta v_l$$belongs to $L^{\tilde{G}}(B_2)$ and it is uniformly bounded in norm. Hence, by H\"{o}lder inequality and the strong convergence of $v_k$, we obtain 
    \begin{equation}
    I_1\to 0 \text{ as }l, k \to \infty.
    \end{equation}A similar reasoning shows that 
    \begin{equation}
    I_2\to 0 \text{ as }l, k \to \infty,
    \end{equation}as well. Next, write
    $$I_3=I_{31}+I_{32}+I_{33},$$where 
    $$I_{31}= \int_{B_2}[g(\Delta v_k)\Delta v_k-g_{\varepsilon_k}(\Delta v_k)\Delta v_k][\Delta v_k-\Delta v_l] \varphi\,dx,$$
     $$I_{32}= \int_{B_2}[g(\Delta v_k)\Delta v_k-g_{\varepsilon_k}(\Delta v_k)\Delta v_k]( v_k- v_l)\Delta \varphi\,dx,$$and finally,
      $$I_{33}= 2\int_{B_2}[g(\Delta v_k)\Delta v_k-g_{\varepsilon_k}(\Delta v_k)\Delta v_k](\nabla v_k-\nabla v_l)\cdot \nabla \varphi\,dx.$$Since \eqref{bound in B m} holds, we obtain by Lemma \ref{G g} again,
      \begin{equation}
      |I_{32}|\leq C(M)\int_{B_2}G(v_k-v_l)\,dx \to 0 \text{ as }k, l \to \infty.
      \end{equation}An analogous reasoning shows that $I_{33}\to 0$ as $k, l\to \infty$. Regarding the term $I_{31}$, in \cite{YZ} it is proved that 
      $$h_\varepsilon(t):=|g(t)-g_\varepsilon(t)|t = O(\varepsilon),$$for $t\in [0, M]$. Therefore, this fact together with Proposition \ref{maximum principle} imply that $I_{31}\to 0$ as $k, l\to \infty$. The term $I_4$ in \eqref{four terms} is treated similarly. Hence, we derive that
      $$\int_{B_{r_m}}\left[g(\Delta v_k)\Delta v_k -g(\Delta v_l)\Delta v_l\right](\Delta v_k-\Delta v_l)\,dx \to 0, \quad \mbox{ as }k, l \to \infty.$$By Lemma \ref{inequality g}, this implies that $\left\lbrace \Delta v_k \right\rbrace$ is a Cauchy sequence in $L^G(B_{r_m})$ and hence $v_k \to v$ strongly in $W^{2, G}(B_{r_m}).$ By considering the balls $B_{r_m}$, the assumption $r_m\to 2$ and extracting a diagonal subsequence, we obtain that
\begin{equation}\label{a e convergence}
\Delta v_k \to \Delta v \text{ a.e. in }B_2.
\end{equation}Hence, taking $k\to \infty$ in \eqref{max princ}, the estimate \eqref{estimate v} follows. Finally,  for any test function $\varphi\in C_0^{\infty}(B_2)$, and letting $K:=supp(\varphi)$, we have by \eqref{max princ}
$$
\sup_{K}G(\Delta v^\varepsilon)\leq C.
$$
Consequently,
$$
|g(\Delta v^\varepsilon)\Delta v^\varepsilon \Delta \varphi| \leq C (G(\Delta v^\varepsilon)+1) \leq C
$$
and by \eqref{a e convergence}
$$
g(\Delta v^\varepsilon)\Delta v^\varepsilon \Delta \varphi \to  g(\Delta v)\Delta v \Delta \varphi \quad \text{a.e. in K as } \varepsilon\to 0^+.
$$
Therefore, by Lebesgue's dominated convergence theorem, $v$ solves \eqref{eq for v}.
\end{proof}
   
\begin{proposition}\label{biarmonic aprox}
Let $u\in W^{2, G}_{loc}(\Omega)$ be a weak solution of \eqref{eq}. For $\varepsilon>0$, there is $\delta=\delta(\varepsilon)>0$ such that if
\begin{equation}\label{assumption f}
\int_{B_2}G(\Delta u)\,dx\leq 1 \text{ and }\int_{B_2}G(f)\,dx \leq \delta,
\end{equation}then there is $v\in W^{2, G}(B_2)$ a weak solution of
\begin{equation}\label{auxiliary eq}
\begin{cases}\Delta\left(g(\Delta v)\Delta v\right)=0, \quad \text{in }B_2\\
v-u\in W_0^{2, G}(B_2), \quad \text{on }\partial B_2.
\end{cases}
\end{equation}such that
\begin{equation}\label{conclusion 1}
\int_{B_2}G(\Delta u -\Delta v)\,dx \leq \varepsilon
\end{equation}and
\begin{equation}\label{conclusion 2}
\sup_{B_{3/2}}G(\Delta v)\leq N_0,
\end{equation}for some $N_0>1$.
\end{proposition} 

\begin{proof}
The estimate \eqref{conclusion 2} follows from \eqref{estimate v} and the assumption \eqref{assumption f}.

Since $u$ is a weak solution of \eqref{eq} and $v$ is a solution of \eqref{auxiliary eq}, we test both equations with $\varphi=u-v \in W_0^{2, G}(B_2)$ and we get
\begin{equation}\label{test with varphi}
\int_{B_2}\left(g(\Delta u)\Delta u-g(\Delta v)\Delta v\right)(\Delta u-\Delta v)\,dx =\int_{B_2}\left(g(f)f\right)(\Delta u-\Delta v)\,dx.
\end{equation}By Lemma \ref{inequality g} and Young inequality \eqref{Young} with $\eta>0$, we get
\begin{equation}\label{ineq 1}
\int_{B_2}\left(g(\Delta u)\Delta u-g(\Delta v)\Delta v\right)(\Delta u-\Delta v)\,dx \geq C\int_{B_2}G(\Delta u-\Delta v)\,dx
\end{equation}and
\begin{equation}\label{ineq 2}
\int_{B_2}\left(g(f)f\right)(\Delta u-\Delta v)\,dx\leq C_\eta\int_{B_2}G(f)\,dx + \eta\int_{B_2}G(\Delta u -\Delta v)\,dx.
\end{equation}Hence, combining \eqref{ineq 1} and \eqref{ineq 2} with \eqref{test with varphi} and appealing to \eqref{assumption f}, we get 
\begin{equation}
C\int_{B_2}G(\Delta u-\Delta v)\,dx\leq \delta C_\eta + \eta\int_{B_2}G(\Delta u -\Delta v)\,dx.
\end{equation}Choosing $\eta < C/2$ and, for $\varepsilon>0$ arbitrary,  $\delta < \varepsilon/ (2C_\eta)$ give \eqref{conclusion 1}. 
\end{proof}  

\subsection{Step 3: estimates for the level sets of the Hardy-Littlewood maximal function}

\begin{proposition}\label{prop auxi}
There exists $N_2>1$ such that for any $\varepsilon>0$, there is $\delta>0$ so that if $u$ is a local weak solution of \eqref{eq} in $\Omega$ with $B_{4r}\subset \Omega$ and
$$|\left\lbrace x\in B_r: \mathcal{M}(G(\Delta u))(x)>N_2\right\rbrace| \geq \varepsilon|B_r|,$$then there holds
$$B_r \subset \left\lbrace x\in B_r: \mathcal{M}(G(\Delta u))(x)>1\right\rbrace \cup \left\lbrace x\in B_r: \mathcal{M}(G(f))(x)>\delta \right\rbrace.$$
\end{proposition}

\begin{proof}
By rescaling
$$u_r(x):=\dfrac{1}{r^2}u(rx), \quad h(x):=f(rx)$$we obtain that $v$ is a local weak solution of
$$\Delta(g(\Delta w )\Delta w)=\Delta (g(h)h),$$in the set $\Omega'=\frac{1}{r}\Omega$. Observe that by assumption $B_4\subset \Omega'$. Therefore, it is enough to show that
$$|\left\lbrace x\in B_1: \mathcal{M}(G(\Delta u))(x)>N_2\right\rbrace| \geq \varepsilon|B_1|,$$implies
$$B_1 \subset \left\lbrace x\in B_1: \mathcal{M}(G(\Delta u))(x)>1\right\rbrace \cup \left\lbrace x\in B_1: \mathcal{M}(G(f))(x)>\delta \right\rbrace.$$

We will prove the counter reciprocal. Assume that $x_1\in B_1$ satisfies
\begin{equation}\label{ineq 13}
\mathcal{M}(G(\Delta u))(x_1)\leq 1 \quad \text{ and } \mathcal{M}(G(f))(x_1)\leq \delta.
\end{equation}Then,
$$\fint_{B_2}G(\Delta u)\,dx = \dfrac{1}{\omega_n2^n}\int_{B_2}G(\Delta u)\,dx \leq \left( \dfrac{3}{2}\right)^n\fint_{B_3(x_1)}G(\Delta u)\,dx \leq \left(\dfrac{3}{2}\right)^n,$$where $\omega_n$ is the measure of the unit ball in $\mathbb{R}^n$. Similarly,
\begin{equation}
\fint_{B_2}G(f)\,dx \leq \left(\dfrac{3}{2}\right)^n\delta.
\end{equation}By Proposition \ref{biarmonic aprox}, for any $\eta>0$, there is $\delta>0$ and $v\in W^{2, G}(B_2)$ a local weak solution of \eqref{eq for v} such that 
\begin{equation}\label{est 12}
\fint_{B_2}G(\Delta u -\Delta v)\,dx  \leq \eta
\end{equation}and
\begin{equation}\label{est 11}
\sup_{B_{3/2}}G(\Delta v)\leq N_0,
\end{equation}for some $N_0>1$. Let 
$$A:=\left\lbrace x\in B_1:\mathcal{M}( G(\Delta u))(x)>N_2\right\rbrace$$and
$$B=\left\lbrace x\in B_1: \mathcal{M}( G(\Delta u-\Delta v))(x)>N_0\right\rbrace,$$for some $N_2>1$ to be chosen. We will show that $A\subset B$. Let 
\begin{equation}\label{cond x 2}
x_2\in \left\lbrace x\in B_1:\mathcal{M}(G(\Delta u -\Delta v))(x)\leq N_0 \right\rbrace.\end{equation}
Then, for $r\le 1/2$, we have $B_r(x_2)\subset B_{3/2}$ and so by \eqref{cond x 2} and  \eqref{est 11}, we obtain
 \begin{equation}
 \begin{split}
 \fint_{B_r(x_2)}G(\Delta u)\,dx &\leq \dfrac{1}{2}\left(\fint_{B_r(x_2)}G(2(\Delta u- \Delta v))\,dx +\fint_{B_r(x_2)}G(2\Delta v)\,dx \right) \\& \leq C\left(\fint_{B_r(x_2)}G(\Delta u- \Delta v)\,dx +\fint_{B_r(x_2)}G(\Delta v)\,dx \right) \\&
 \leq CN_0,
 \end{split}
 \end{equation}
 If now $r>1/2$, then 
 $$
 \fint_{B_r(x_2)}G(\Delta u)\,dx \leq 5^n \fint_{B_{5r}(x_1)}G(\Delta u)\,dx \leq 5^n,
 $$
 by \eqref{ineq 13}. Therefore, we conclude that there is $N_2>1$ such that $x_2 \notin A$. Thus, $A\subset B$.  Hence, 
 \begin{equation*}
 \begin{split}
& |\left\lbrace x \in B_1:\mathcal{M}(G(\Delta u))(x)>N_2 \right\rbrace|\\ & \qquad = |A| \leq |B|\\&  \qquad =|\left\lbrace x\in B_1: \mathcal{M}(G(\Delta u-\Delta v))(x)>N_0\right\rbrace|\\& \qquad \leq C\int_{B_1}G(\Delta u-\Delta v)\,dx \leq C\eta \leq \varepsilon,
 \end{split}
 \end{equation*}taking $\eta$ small enough. This ends the proof of the proposition. 
\end{proof}

\begin{proposition}
Assume that $u$ is a local weak solution of \eqref{eq} and let $\delta$ and $N_2$ as in the previous proposition. Assume that
\begin{equation}
|\left\lbrace x\in B_1: \mathcal{M}(G(\Delta u))(x)>N_2 \right\rbrace|<\varepsilon |B_1|.
\end{equation}Then, for any $\lambda>0$, there holds
\begin{equation}\label{tbp}
\begin{split}
&|\left\lbrace x\in B_1: \mathcal{M}(G(\Delta u))(x)>\lambda N_2 \right\rbrace|\\&\qquad \leq 10^n\varepsilon\left(|\left\lbrace x\in B_1: \mathcal{M}(G(\Delta u))(x)> \lambda \right\rbrace| + |\left\lbrace x\in B_1: \mathcal{M}(G(f))(x)>\lambda \delta \right\rbrace|\right).
\end{split}
\end{equation}
\end{proposition}
 
\begin{proof}
The proof follows as in \cite{W} and \cite[Lemma 2.15]{YZ}, by applying the Vitali Lemma \ref{vitali}. Indeed, consider first the case $\lambda=1$ and define the sets
$$C=\left\lbrace x \in B_1: \mathcal{M}(G(\Delta u))(x) > N_2 \right\rbrace$$
 and
 $$D:=\left\lbrace x\in B_1: \mathcal{M}(G(\Delta u))(x)>1 \right\rbrace \cup \left\lbrace x\in B_1: \mathcal{M}(G(f))(x)>\delta \right\rbrace.$$Clearly,
 $$
 C\subset D\subset B_1
$$
and by assumption,
$$
|C|<\varepsilon |B_1|.
$$
Hence, by Proposition \ref{prop auxi}, we may apply Lemma \ref{vitali} to obtain that inequality \eqref{tbp} holds. The general case for any $\lambda>0$ follows by considering that $u$ is a weak solutions of
 $$\dfrac{1}{\lambda}\Delta (g(\Delta u)\Delta u)= \dfrac{1}{\lambda}\Delta (g(f)f),$$and hence there is $N_2>1$ such that for any $\varepsilon>0$ and $r>1$, there is $\delta>0$ such that
 $$\bigg|\left\lbrace x\in B_r: \dfrac{1}{\lambda}\mathcal{M}(G(\Delta u))(x)>N_2\right\rbrace\bigg| \geq \varepsilon|B_r|,$$then there holds
$$B_r \subset \left\lbrace x\in B_r: \dfrac{1}{\lambda}\mathcal{M}(G(\Delta u))(x)>1\right\rbrace \cup \left\lbrace x\in B_r: \dfrac{1}{\lambda}\mathcal{M}(G(f))(x)>\delta \right\rbrace$$
and the proof is complete.
\end{proof}

In the final step, we will finish the proof of Theorem \ref{main theorem}.

\subsection*{Step 4: proof of Theorem \ref{main theorem}} For $\delta>0$ small to be chosen, let
$$\lambda_0:= \dfrac{1}{\delta}\left\lbrace  \int_{B_2}G(u)\,dx +\int_{B_2}G(|\nabla u|)\,dx+ \left(\int_{B_2}G(f)^q \,dx\right)^{1/q}\right\rbrace.$$Then, by Theorem \ref{theorem q 1}
\begin{equation}
\int_{B_1}\dfrac{1}{\lambda_0}G(\Delta u)\,dx \leq C\left(\int_{B_2}\dfrac{1}{\lambda_0}G(u)\,dx+\int_{B_2}\dfrac{1}{\lambda_0}G(|\nabla u|)\,dx +\int_{B_2}\dfrac{1}{\lambda_0}G(f)\,dx\right).
\end{equation}Therefore, by H\"{o}lder's inequality,
\begin{equation}\label{first calc}
\begin{split}
\fint_{B_1}\dfrac{1}{\lambda_0}G(\Delta u)\,dx  & \leq  C\left(\fint_{B_2}\dfrac{1}{\lambda_0}G(u)\,dx+\fint_{B_2}\dfrac{1}{\lambda_0}G(|\nabla u|)\,dx +\fint_{B_2}\dfrac{1}{\lambda_0}G(f)\,dx \right)\\& \leq C\left(\delta + \dfrac{1}{\lambda_0}\left(\fint_{B_2}G(f)^q\,dx\right)^{1/q}\right) \leq C\delta.
\end{split}
\end{equation}Moreover, 
\begin{equation}
\left|\left\lbrace x\in B_1 : \dfrac{1}{\lambda_0}\mathcal{M}(G(\Delta u))(x)>N_2\right\rbrace\right|\leq \dfrac{1}{N_2}\int_{B_1}\dfrac{1}{\lambda_0}G(\Delta u)\,dx \leq \varepsilon|B_1|,
\end{equation}choosing $\delta<\varepsilon N_2/C$ in \eqref{first calc}. Also, by definition of $\lambda_0$,
\begin{equation}\label{second calc}
\int_{B_1}\dfrac{1}{\lambda_0^q}G(f)^q\,dx \leq C\delta^q.
\end{equation}Hence, by the previous proposition, we derive
\begin{equation}
\begin{split}
&\bigg|\left\lbrace x\in B_1: \dfrac{1}{\lambda_0}\mathcal{M}(G(\Delta u))(x)>\lambda N_2 \right\rbrace\bigg|\\&\qquad \leq 10^n\varepsilon\left(\bigg|\left\lbrace x\in B_1: \dfrac{1}{\lambda_0}\mathcal{M}(G(\Delta u))(x)> \lambda \right\rbrace\bigg| + \bigg|\left\lbrace x\in B_1: \dfrac{1}{\lambda_0}\mathcal{M}(G(f))(x)>\lambda \delta^q \right\rbrace\bigg|\right).
\end{split}
\end{equation}Thus, as $q>1$,we deduce as in \cite[page 1550]{YZ}, that
\begin{equation}
\int_{B_1}\left(\dfrac{1}{\lambda_0}\mathcal{M}(G(\Delta u)) \right)^q\,dx \leq C_1\varepsilon\int_{B_1}\left(\dfrac{1}{\lambda_0}\mathcal{M}(G(\Delta u)) \right)^q\,dx+C_2\int_{B_1}\left(\dfrac{1}{\lambda_0}G(f) \right)^q\,dx.
\end{equation} So, choosing $\varepsilon$ small and recalling \eqref{second calc}, we obtain

$$\int_{B_1}\left(\dfrac{1}{\lambda_0}\mathcal{M}(G(\Delta u)) \right)^q\,dx \leq C$$which certainly implies 
$$\int_{B_1}\left(\dfrac{1}{\lambda_0}G(\Delta u) \right)^q\,dx \leq C$$and consequently,
$$\int_{B_1}G(\Delta u)^q\,dx \leq C\left\lbrace\left(\int_{B_2}G(u)\,dx + \int_{B_2}G(|\nabla u|)\,dx\right)^q + \int_{B_2}G(f)^q\,dx \right\rbrace.$$

This completes the proof of the Theorem.\qed

    \end{document}